\documentclass[letterpaper,12pt,normalheadings]{scrartcl}
\usepackage[american]{babel}
\usepackage[utf8]{inputenc}
\usepackage[T1]{fontenc}
\usepackage[shortcuts]{extdash}
\usepackage{geometry}
\usepackage{hyperref}
\hypersetup{colorlinks=true, linkcolor=chroma4blue, citecolor=chroma4blue, urlcolor=black, filecolor=black}
\usepackage{lmodern}
\usepackage{parskip}
\usepackage{caption}
\usepackage{subcaption}
\usepackage{authblk}
\usepackage{graphicx}
\usepackage{amsmath}
\usepackage{amssymb}
\usepackage{amsthm}
\usepackage{dsfont}
\usepackage{mathtools}
\usepackage{booktabs}
\usepackage{tabulary}
\usepackage{xspace}
\usepackage{tikz}

\tikzstyle{vertex} = [draw=black, fill=black, inner sep=0.55mm, circle]
\tikzstyle{edge} = [black, thick]

\usetikzlibrary{calc}
\usetikzlibrary{positioning}
\newcommand{\labelLeft}[2]{\node[left=.05cm of #1] {#2};}
\newcommand{\labelRight}[2]{\node[right=.05cm of #1] {#2};}
\newcommand{\labelAbove}[2]{\node[above=.05cm of #1] {#2};}
\newcommand{\labelBelow}[2]{\node[below=.05cm of #1] {#2};}

\usetikzlibrary{backgrounds}
\pgfdeclarelayer{background}
\pgfsetlayers{background, main}

\definecolor{chroma4blue}{RGB}{63,97,135}
\definecolor{chroma4green}{RGB}{65,143,126}
\definecolor{chroma4sand}{RGB}{171,166,125}
\definecolor{chroma4gray}{RGB}{199,199,199}

\definecolor{chroma5blue}{RGB}{63, 97, 135}
\definecolor{chroma5green}{RGB}{64,130,129}
\definecolor{chroma5grass}{RGB}{116,154,126}
\definecolor{chroma5sand}{RGB}{177,173,140}
\definecolor{chroma5gray}{RGB}{199,199,199}

\mathcode`@="8000
{\catcode`\@=\active\gdef@{\mkern1mu}}

\DeclareMathOperator{\bridge}{\oplus}
\DeclareMathOperator{\spokeStandAlone}{\vcenter{\hbox{\scalebox{1}{\rotatebox{90}{$\oslash$}}}}}
\DeclareMathOperator{\spoke}{\spokeStandAlone@@}

\newcommand{\kvertexConnected}[1]{uniformly \mbox{${#1}$-c}on\-nect\-ed\xspace}

\newcommand{\kedgeConnected}[1]{uniformly \mbox{${#1}$-e}dge-con\-nect\-ed\xspace}
\newcommand{\kEdgeConnected}[1]{Uniformly \mbox{${#1}$-e}dge-con\-nect\-ed\xspace}
\newcommand{\vertexConnected}{uniformly connected\xspace}

\newcommand{\edgeConnected}{uniformly edge-connected\xspace}

\makeatletter
\newcommand{\pushright}[1]{\ifmeasuring@#1\else\omit\hfill$\displaystyle#1$\fi\ignorespaces}
\newcommand{\pushleft}[1]{\ifmeasuring@#1\else\omit$\displaystyle#1$\hfill\fi\ignorespaces}
\makeatother

\makeatletter
\newlength{\negph@wd}
\DeclareRobustCommand{\negphantom}[1]{%
  \ifmmode
    \mathpalette\negph@math{#1}%
  \else
    \negph@do{#1}%
  \fi
}
\newcommand{\negph@math}[2]{\negph@do{$\m@th#1#2$}}
\newcommand{\negph@do}[1]{%
  \settowidth{\negph@wd}{#1}%
  \hspace*{-\negph@wd}%
}
\makeatother

\theoremstyle{definition}

\newtheorem{theorem}{Theorem}
\newtheorem{definition}[theorem]{Definition}
\newtheorem{corollary}[theorem]{Corollary}
\newtheorem{lemma}[theorem]{Lemma}

\addtokomafont{sectioning}{\rmfamily}

\title{\Large Uniformly connected graphs}
\author{\normalsize Frank Göring\textsuperscript{1}, Tobias Hofmann\textsuperscript{1,}*, Manuel Streicher\textsuperscript{2}}
\date{} % \normalsize \today
\affil{\footnotesize
\textsuperscript{1}Faculty of Mathematics, Chemnitz University of Technology, Germany\\
\textsuperscript{2}Department of Mathematics, Technische Universität Kaiserslautern, Germany\\
*Corresponding Author, \texttt{tobias.hofmann@math.tu-chemnitz.de}}

\begin{document}
\maketitle

\begin{abstract}\noindent
\textbf{Abstract.} In this article we investigate the structure of \kvertexConnected{k} and \kedgeConnected{k} graphs. Whereas both types have previously been studied independent of each other, we analyze relations between these two classes. We prove that any \kvertexConnected{k} graph is also \kedgeConnected{k} for $k\le 3$ and demonstrate that this is not the case for $k>3$. Furthermore, \kvertexConnected{k} and \kedgeConnected{k} graphs are well understood for $k\le 2$ and it is known how to construct \kedgeConnected{3} graphs. We contribute here a constructive characterization of \kvertexConnected{3} graphs that is inspired by Tuttes Wheel Theorem. Eventually, these results help us to prove a tight bound on the number of vertices of minimum degree in \kvertexConnected{3} graphs.\\


 
\end{abstract}

\section{Introduction}
The aim of this article is to study the structure of \vertexConnected and \edgeConnected graphs. Throughout this article all graphs are finite, undirected, and loopless, however, they may contain parallel edges. Whereas we refer to Diestel~\cite{diestel2017graph} for terminology that is not defined here, we summarize the most relevant concepts and notations at the end of this section.

A graph is said to be \emph{\kvertexConnected{k}} if between any pair of vertices the maximum number of independent paths is exactly $k\in\mathbb{N}$. Uniform edge-connectivity is defined analogously. The main contributions of this article are a proof that all \kvertexConnected{3} graphs are \kedgeConnected{3}, a constructive characterization of all \kvertexConnected{3} graphs, and a tight bound on the number of vertices of minimum degree in \kvertexConnected{3} graphs. 

Uniform connectivity was introduced by Beineke, Oellermann, and Pippert in~\cite{beineke2002average}. In their contribution they derive some basic properties of uniformly connected graphs and discuss relations to other connectivity concepts. Furthermore, they define different operations that, when applied to \vertexConnected graphs, sustain that connectivity. Although their operations yield infinite families of \vertexConnected graphs they by no means construct all such graphs. In contrast, here, we provide operations that preserve uniform $3$-connectivity and additionally can be used to obtain \emph{all} \kvertexConnected{3} graphs. Constructive characterizations are widespread in the study of connectivity concepts. Tutte characterizes all $3$-connected graphs in~\cite{tutte1961three} with his famous \emph{Wheel Theorem}. A different version of the wheel theorem characterizes all cubic, $3$-connected graphs, cf.~\cite{tutte1966connectivity}. Our characterization of \kvertexConnected{3} graphs is in a sense situated between those by Tutte as any \kvertexConnected{3} graph is $3$-connected and any cubic, $3$-connected graphs is \kvertexConnected{3}.

We illustrate the applicability of our characterization by using it to obtain a tight bound on the number of vertices of minimum degree in \kvertexConnected{3} graphs. Various such bounds for different graph classes can be found in the literature. Back in 1969, Halin~\cite{halin1969theorem} showed that there is a vertex of degree~$k$ in each minimally \mbox{$k$-connected} graph. For a graph $G$ on $n$ vertices Dirac~\cite{dirac1967minimally} showed that there are as many as \hbox{$(n+4)/3$} vertices of degree $2$ in minimally $2$-connected graphs and Halin~\cite{halin1969theorem} established that there are at least \hbox{$(2\,n+6)/5$} vertices of degree $3$ in minimally 3-connected graphs. Mader~\cite{mader1979struktur} finally gave a tight bound for general~$k$, which comprises the above cases for $k=2$ and $k=3$. While these bounds are best possible in terms of the number of vertices and~$k$, Oxley~\cite{oxley1981connectivity} used a graph's number of edges as an additional parameter to achieve even stronger bounds, which are however not tight. Finally, the long time open problem of finding a tight lower bound for general~$k$ that depends on the number of vertices and the number of edges has been resolved by Schmidt~\cite{schmidt2018tight}. All these bounds also apply to \kvertexConnected{k} graphs because they are minimally $k$-connected. However, as there are minimally $k$-connected graphs that are not \kvertexConnected{k}, we may aim for even stronger bounds. In particular we prove that \hbox{$(2n+2)/3$} is a tight lower bound for the number of vertices of minimum degree in \kvertexConnected{3} graphs.

Uniform edge-connectivity was introduced by Kingsford and Mar\c{c}ais in~\cite{kingsford2009synthesis} using the term \emph{exact} edge-connectivity. Their main contribution is a nice characterization of all \kedgeConnected{3} graphs. The authors prove basic properties and relate uniform edge-connectivity to other connectivity and regularity \mbox{concepts}. Our main contribution towards uniform edge-connectivity is to prove that any \kvertexConnected{3} graph is also \kedgeConnected{3}. We show this by proving the remarkable fact that a simple, $3$-connected graph in which two vertices are connected by four edge-disjoint paths also contains two vertices that are connected by four independent paths. The same holds true for $k=1$ and $k=2$, but is not true for $k=4$ as the graph in Figure~\ref{fig:vertexRegularNotEdgeRegular} shows.

Another motivation to study uniform connectivity of graphs comes from a branch of spectral graph theory. In that field, \emph{connectivity} and \emph{edge-connectivity matrices} are investigated. For a graph~$G$ and two vertices~$v,w\in V(G)$ an off-diagonal \mbox{$v$-$w$}~entry of such a matrix is the weight of a minimum \mbox{$v$-$w$}~separator, or minimum \mbox{$v$-$w$}~cut, respectively. The diagonal entries are defined as zero. The research on such matrices is a classical topic of combinatorial optimization since at least the seminal work of Gomory and Hu~\cite{gomoryhu1961multi}. Recently, related spectral properties received further attention. As a key topic in spectral graph theory, also the question arises for which graphs certain spectral parameters are extremal. In the article~\cite{hofmann2021}, Hofmann and Schwerdtfeger provide a tight bound on the \emph{energy} of an edge-connectivity matrix, that is the sum of the absolute values of the eigenvalues. In fact, this bound is attained for \edgeConnected graphs, which actually inspired our work on uniform connectivity.

\textbf{Outline.} In Section~\ref{sec:basic}, we examine basic properties of \kvertexConnected{k} and \kedgeConnected{k} graphs. At this point, we also take a closer look at the case~$k=2$. This leads us to Section~\ref{sec:inclusion}, where we prove the fact that for~$k\le 3$ any \kvertexConnected{k} graph is also \kedgeConnected{k}. We provide a constructive characterization of \kvertexConnected{3} graphs in Section~\ref{sec:construction}. We use these construction ideas in Section~\ref{sec:degrees} to establish bounds on the number of vertices of minimum degree and demonstrate that the obtained results can in general not be improved.

The remainder of this section recalls certain notions that are particularly important for our investigation or differ from the notation in Diestel~\cite{diestel2017graph}. Let $G$ be a graph. As we allow parallel edges, we may not identify a given edge with its endvertices. However, by abuse of notation, we refer to any particular edge joining two vertices $v,w\in V(G)$ by $vw$. If the particular choice of the edge is of relevance further comments are made accordingly. For two sets of vertices~$S,T\subset V(G)$, we denote by $E(S,T):=\{vw\in E(G):v\in S\text{ and }w\in T\}$ the edges joining vertices from $S$ and $T$. For two graphs $G$ and $H$, we define the graph $G\cup H$ as the graph with vertex set $V(G)\cup V(H)$ and edge set $E(G)\cup E(H)$. The graph $G$ is called \emph{connected} if any two vertices of $G$ are connected by a path. We refer to a maximal connected subgraph of $G$ as a \emph{component}. For two sets $A,B\subseteq V(G)$, a set $X\subseteq V(G)\cup E(G)$ \emph{separates} $A$ and $B$ if any $A$-$B$ path contains an element from $X$. The set $X$ separates two vertices $v,w\in V(G)$ if $v, w\notin X$ and $X$ separates $\{v\}$ and $\{w\}$. A subset $S\subseteq V(G)$ is a \emph{separator} if it separates two vertices $v,w\in V(G)$. In this case we say $S$ is a $v$-$w$ separator. A \emph{cutvertex} is a separator consisting of a single vertex. For $k\in\mathbb{N}$ the graph $G$ is \emph{$k$-connected} if $|V(G)|\geq k+1$ and $G-S$ is connected for any set $S$ with $|S|\leq k-1$. A \emph{cut} in $G$ is an edge set $E(S,V(G)\setminus S)$ where $S$ is a nonempty proper subset of $V(G)$. We refer to $S$ and $V(G)\setminus S$ as the \emph{sides} of the cut. For two vertices~$v,w\in V(G)$, a \emph{$v$-$w$ cut} is a cut in $G$ such that $v$ and $w$ are in different sides of the cut. A \emph{bridge} is a cut that contains exactly one edge. For $k\in\mathbb{N}$ the graph $G$ is called $k$-edge-connected if $G-F$ is connected for any set $F\subseteq E(G)$ with $|F|\leq k-1$. In order to shorten notation we use the notations $k$-cut or $k$-separator to indicate that the corresponding set contains $k\in\mathbb{N}$ elements. A maximal $2$-connected subgraph of a graph is called \emph{block}. We say two or more paths are \emph{independent} if every vertex that is contained in more than one path is an endpoint of all paths it is contained in. For a set $X\subseteq V(G)$ and a vertex $v\in V(G)\setminus X$ we call the set of some $v$-$X$ paths a \emph{$v$-$X$ fan} if any two of the paths only have the vertex $v$ in common. 

\section{Basic properties}\label{sec:basic}
This section summarizes basic structural results about \kvertexConnected{k} as well as \kedgeConnected{k} graphs and provides concise characterizations for $k\le 2$.
\begin{definition}\label{def:vertexConnected}
Let~$G$ be a graph with at least $k+1$ vertices. We call~$G$ \emph{\kvertexConnected{k}} if $k\in\mathbb{N}$ is the maximum number of independent paths between any two vertices in $V(G)$.
\end{definition}
\begin{definition}\label{def:edgeConnected}
Let~$G$ be a graph with at least two vertices. We call~$G$ \emph{\kedgeConnected{k}} if $k\in\mathbb{N}$ is the maximum number of edge-disjoint paths between any two vertices in $V(G)$.
\end{definition}
In situations where it is dispensable, we may omit the parameter~$k$ and simply use the terms \emph{\vertexConnected} or \emph{\edgeConnected}. Although it may not be apparent at first sight, Definition~\ref{def:vertexConnected} only comprises simple graphs.
\begin{lemma}\label{lem:parallelEdges}
If two vertices $v$ and $w$ of a $k$-connected graph are joined by parallel edges, then there are at least $k+1$ independent paths between $v$ and $w$.
\end{lemma}
\begin{proof}
We note first that a $k$-connected graph contains at least $k+1$ vertices by definition. Let us suppose, for the sake of contradiction, that $v$ and $w$ are connected by at most $k$ independent paths in $G$. Denote by $G^\prime$ the graph where all edges joining $v$ and $w$ are removed from $G$. Then the vertices $v$ and $w$ are connected by at most $k-2$ independent paths in $G^\prime$. By Menger's Theorem, there is a set $S\subset V(G)\setminus\{v,w\}$ with at most $k-2$ vertices such that $v$ and $w$ are in different components of $G^\prime-S$. Since $G$ has at least $k+1$ vertices, either $G^\prime-S$ has at least three components or one of the components contains at least two vertices. Thus, $S\cup\{v\}$ or $S\cup\{w\}$ is a separator in $G$ containing at most $k-1$ vertices. This is a contradiction to $G$ being $k$-connected.
\end{proof}
As a direct consequence of Lemma~\ref{lem:parallelEdges}, we obtain the following statement about \vertexConnected graphs.
\begin{lemma}\label{lem:vertexRegularSimple}
A \kvertexConnected{k} graph does not have parallel edges.
\end{lemma}

\begin{figure}
\centering
\begin{tikzpicture}[scale=2]
    \node[draw=black, fill=black, inner sep=0.55mm, circle] (1) at (0,1) { };
    \node[draw=black, fill=black, inner sep=0.55mm, circle] (2) at (0,0) { };
    \node[draw=black, fill=black, inner sep=0.55mm, circle] (3) at (1,0) { };
    \node[draw=black, fill=black, inner sep=0.55mm, circle] (4) at (1,1) { };
    \draw[black, thick] (1) to (2);
    \draw[black, thick] (1) to (3);
    \draw[black, thick] (1) to (4);
\end{tikzpicture}\hspace{5ex}
\begin{tikzpicture}[scale=2]
    \node[draw=black, fill=black, inner sep=0.55mm, circle] (1) at (0,1) { };
    \node[draw=black, fill=black, inner sep=0.55mm, circle] (2) at (0,0) { };
    \node[draw=black, fill=black, inner sep=0.55mm, circle] (3) at (1,0) { };
    \node[draw=black, fill=black, inner sep=0.55mm, circle] (4) at (1,1) { };
    \draw[black, thick] (1) to (2) to (3) to (4) to (1);
\end{tikzpicture}\hspace{5ex}
\begin{tikzpicture}[scale=2]
    \node[draw=black, fill=black, inner sep=0.55mm, circle] (1) at (0,1) { };
    \node[draw=black, fill=black, inner sep=0.55mm, circle] (2) at (0,0) { };
    \node[draw=black, fill=black, inner sep=0.55mm, circle] (3) at (1,0) { };
    \node[draw=black, fill=black, inner sep=0.55mm, circle] (4) at (1,1) { };
    \draw[black, thick] (1) to (2) to (3) to (4) to (1);
    \draw[black, thick] (1) to (3);
    \draw[black, thick] (2) to (4);
\end{tikzpicture}\hspace{5ex}
\begin{tikzpicture}[scale=2]
    \node[draw=black, fill=black, inner sep=0.55mm, circle] (1) at (0,1) { };
    \node[draw=black, fill=black, inner sep=0.55mm, circle] (2) at (0,0) { };
    \node[draw=black, fill=black, inner sep=0.55mm, circle] (3) at (1,0) { };
    \node[draw=black, fill=black, inner sep=0.55mm, circle] (4) at (1,1) { };
    \node[draw=black, fill=black, inner sep=0.55mm, circle] (5) at (0.5,0.5) { };
    \draw[black, thick] (2) to (3);
    \draw[black, thick] (4) to (1);
    \draw[black, thick] (1) to (5);
    \draw[black, thick] (2) to (5);
    \draw[black, thick] (3) to (5);
    \draw[black, thick] (4) to (5);
\end{tikzpicture}
\caption{Small \edgeConnected graphs.}
\label{fig:smallexamples}
\end{figure}
For a first impression, Figure~\ref{fig:smallexamples} shows a few small examples of \edgeConnected graphs. We observe that a graph~$G$ with at least two vertices is \kvertexConnected{1} if and only if~$G$ is \kedgeConnected{1} if and only if~$G$ is a tree. This immediately follows from the respective definitions. One may also consider \kvertexConnected{0} graphs, which are just graphs without edges. We note that there are \kedgeConnected{2} graphs that are not \vertexConnected. The smallest such graph without parallel edges is the hourglass graph on the right in Figure~\ref{fig:smallexamples}. The other three graphs in the figure are both \edgeConnected and \vertexConnected. The question for a \vertexConnected graph that is not \edgeConnected remains. It is addressed in Section~\ref{sec:inclusion}. We also observe that \kvertexConnected{k} graphs are $k$-connected and \kedgeConnected{k} graphs are $k$-edge-connected. Another simple, but useful observation relates both of our classes.
\begin{lemma}\label{lem:kEdgeConnectedIsVertexConnected}
Each $k$-connected graph that is \kedgeConnected{k} is  \kvertexConnected{k}.
\end{lemma}
\begin{proof}
By Menger's Theorem, there are at least $k$ independent paths between any two vertices in a $k$-connected graph. Furthermore, we are given a \kedgeConnected{k} graph. So there are not more than $k$ independent paths between two vertices, as such paths are edge-disjoint as well. This proves the statement.
\end{proof}
\newpage
Moreover, Menger's Theorem~\cite{menger1927allgemeinen} provides us with the following alternative perspectives on our classes. A graph~$G$ is \kedgeConnected{k} if a minimum $v$-$w$ cut has cardinality~$k$ for any two vertices in $V(G)$. Likewise, a graph $G$ is \kvertexConnected{k} if it is a graph on at least $k+1$ vertices without parallel edges and a minimum $v$-$w$ separator has cardinality~$k\in\mathbb{N}$ for any two nonadjacent vertices in $V(G)$ and cardinality~$k-1$ for any two adjacent vertices in $V(G)$. The underlying edge and vertex versions of Mengers' Theorem are treated by Diestel in~\cite[Section 3]{diestel2017graph}.

Beineke, Oellermann, and Pippert~\cite{beineke2002average} show that \kvertexConnected{k} graphs are minimally $k$-connected for $k\ge 1$ and critically $k$-connected for $k\ge 2$. Their argumentation can also be adapted to the edge case. \kEdgeConnected{k} graphs are minimally $k$-edge-connected for $k\ge 1$ and critically $k$-edge-connected for $k\ge 2$. The statements about critically connected graphs are formulated only for $k\ge2$ as trees are both \kedgeConnected{1} as well as \kvertexConnected{1}, but neither critically 1-connected nor critically 1-edge-connected, because trees clearly retain their connectivity when deleting leaves. Furthermore, \vertexConnected as well as \edgeConnected graphs can be recognized in polynomial time as the maximum number of independent and edge-disjoint paths can be calculated by the flow-based methods of Even and Tarjan~\cite{even1975network}. More recently, Preißner and Schmidt~\cite{preisser2020computing} proposed even linear time algorithms for such tasks, which can however not be applied to each pair of vertices of a graph.

Our notion of uniform connectivity may also be regarded as a regularity concept as it is concerned with a relation that equally involves all pairs of nodes. The examples from Figure~\ref{fig:smallexamples} show that neither \vertexConnected nor \edgeConnected graphs have to be regular by degree. However, Menger's Theorem implies that if we are given a $k$-regular graph~$G$, then $G$ is \kvertexConnected{k} if and only if~$G$ is simple and~$k$-connected and $G$ is \kedgeConnected{k} if and only if~$G$ is~$k$-edge-connected. This tells us, for instance, that the edge graphs of all simple $k$-dimensional polytopes are \kvertexConnected{k}, as they are $k$-connected by Balinski's Theorem~\cite{balinski1961graph}. We also observe a connection to the notion of distance regular graphs, which are comprehensively examined by Brouwer, Cohen, and Neumaier~\cite{brouwer2012distance}. Recall that distance regular graphs are regular by definition and Brouwer and Koolen~\cite{brouwer2009vertex} show that distance regular graphs having vertices of degree $k$ are $k$-connected and thus also $k$-edge-connected. So distance regular graphs are both \vertexConnected and \edgeConnected. On the other hand, since neither \vertexConnected nor \edgeConnected graphs have to be regular, they also do not have to be distance regular. Further notions that by name might be related to our classes are those of path-regular graphs, introduced by Matula and Dolev~\cite{matula1980path}, or those of $k$-uniformly connected graphs, surveyed by~Chartrand and Zhang~\cite{chartrand2020uniformly}. It is, however, not too difficult to see that the respective classes do not include each other.

To the end of this section, we summarize results about the structure of \kvertexConnected{2} and \kedgeConnected{2} graphs. Beineke, Oellermann, and Pippert~\cite{beineke2002average} observe that a graph is \kvertexConnected{k} if and only if it is $k$-connected and contains no subgraph homeomorphic to the tripartite graph $K_{1,1,2}$. This means that a graph is \kvertexConnected{2} if and only if it is $2$-connected and none of its cycles contains a chord. In other words, a graph is \kvertexConnected{2} if and only if it is a cycle. Kingsford and Mar{\c{c}}ais~\cite{kingsford2009vertices} show that a graph is \kedgeConnected{2} if and only if it is connected and each of its blocks is a cycle. More generally, they show for any $k\in\mathbb{N}$ that a graph which is obtained by gluing two \kedgeConnected{k} graphs at a single vertex remains \kedgeConnected{k}. For $k\geq3$, however, this gluing operation is not enough to generate all \kedgeConnected{k} graphs.

\section{Inclusions between \vertexConnected and \edgeConnected graphs}\label{sec:inclusion}

This section focuses on proving the remarkable fact that any \kvertexConnected{k} graph is also \kedgeConnected{k} for $k\le 3$. For $k=1$ both classes comprise by definition exactly all trees. The case $k=2$ follows directly from the characterizations from the end of Section~\ref{sec:basic}. For $k=3$ this statement is not as obvious and the inclusion to be shown does not hold for $k=4$ as is demonstrated by the example from Figure~\ref{fig:vertexRegularNotEdgeRegular}. Whereas the depicted graph contains vertices of degree $4$, the non-solid lines there indicate five edge-disjoint paths between $v$ and~$w$. Thus, the graph in Figure~\ref{fig:vertexRegularNotEdgeRegular} is not \kedgeConnected{4}. However, note that $v$ and~$w$ are the only vertices of degree larger than $4$ and we easily find a separator with three vertices when deleting the edge~$vw$. Consequently, the number of independent paths between $v$ and~$w$ is four. Furthermore, we can observe that our graph is $4$-connected and thereby is indeed \kvertexConnected{4}. It is also not too difficult to verify algorithmically that our example is the smallest of its kind by enumerating all graphs on at most seven vertices completely. Our proof of the central statement of this section uses the following fan version of Menger's Theorem, which can be found in Diestel~\cite[Section 3.3]{diestel2017graph}.

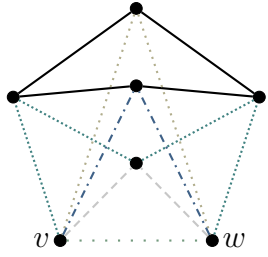
\begin{figure}
\centering
\begin{tikzpicture}[scale=2]
    \node[draw=black, fill=black, inner sep=0.55mm, circle, label={[label distance=-0.5ex]180:$v$}] (1) at (-0.5,-0.688) { };
    \node[draw=black, fill=black, inner sep=0.55mm, circle, label={[label distance=-0.5ex]0:$w$}] (2) at (0.5,-0.688) { };
    \node[draw=black, fill=black, inner sep=0.55mm, circle] (3) at (0.809,0.263) { };
    \node[draw=black, fill=black, inner sep=0.55mm, circle] (4) at (0,0.851) { };
    \node[draw=black, fill=black, inner sep=0.55mm, circle] (5) at (-0.809,0.263) { };
    \node[draw=black, fill=black, inner sep=0.55mm, circle] (6) at (0,0.338) { };
    \node[draw=black, fill=black, inner sep=0.55mm, circle] (7) at (0,-0.175) { };

	\draw[black, thick] (3) -- (4) -- (5) -- (6) -- (3);
	
	\draw[chroma5grass, loosely dotted, thick] (1) -- (2);
	\draw[chroma5sand, thick, dotted] (4) -- (1);
	\draw[chroma5sand, thick, dotted] (4) -- (2);
	\draw[chroma5blue, thick, dashdotted] (1) -- (6);
	\draw[chroma5blue, thick, dashdotted] (2) -- (6);
	\draw[chroma5gray, thick, densely dashed] (1) -- (7);
	\draw[chroma5gray, thick, densely dashed] (2) -- (7);
	\draw[chroma5green, thick, densely dotted] (1) -- (5);
	\draw[chroma5green, thick, densely dotted] (2) -- (3);
	\draw[chroma5green, thick, densely dotted] (7) -- (5);
	\draw[chroma5green, thick, densely dotted] (7) -- (3);
	\end{tikzpicture}
    \caption{A \vertexConnected graph that is not \edgeConnected.}
    \label{fig:vertexRegularNotEdgeRegular}
\end{figure}

\begin{theorem}\label{thm:fanMenger}
Given a graph $G$, a vertex set $X\subset V(G)$, and a vertex $v\in V(G)\setminus X$. Then the minimum number of vertices separating $v$ from $X$ in $G$ is equal to the maximum number of
paths forming a $v$-$X$ fan in $G$.
\end{theorem}
\newpage
Before our central statement, we prove a lemma that makes part of the subsequent arguments more concise. The lemma is a consequence of Theorem~\ref{thm:fanMenger} and ensures the existence of certain fans between vertices of a $k$-separator in a $k$-connected~graph.
\begin{lemma}\label{lem:fanInKseparator}
For $k\in\mathbb{N}$ with $k>0$, let $G$ be a $k$-connected graph, and consider two distinct vertices $v,w\in V(G)$. Choose $G_v\subseteq G-w$ and $G_w\subseteq G-v$ such that $G_v\cup G_w = G$ and $S=V(G_v)\cap V(G_w)$ satisfies $|S|=k$.

Then, for each vertex $x\in S$, there exists an $x$-$(S\setminus\{x\})$ fan in $G_w$ consisting of $\min\{|N_{G_w}(x)|, k-1\}$ paths.
\end{lemma}
\begin{proof}
Let $T\subseteq V(G_w)\setminus\{x\}$ be a set that separates $x$ and $S\setminus\{x\}$ in $G_w$.
If $T$ contains less vertices than $N_{G_w}(x)$, then $T$ also separates a vertex $y\in N_{G_w}(x)\setminus T$ from $S\setminus\{x\}$. Since $G$ is $k$-connected, by Theorem~\ref{thm:fanMenger}, there exists a $y$-$S$ fan $F$ consisting of $k$ paths in $G$. Each of the paths of the fan $F$ ends in a distinct vertex of $S$ and $F$ is indeed contained in $G_w$. Recalling that $T$ separates $y$ from $S\setminus\{x\}$, we obtain that $|T|\geq k-1$.

We showed that any set $T\subseteq V(G_w)\setminus\{x\}$ that separates $x$ and $S\setminus\{x\}$ in $G_w$ consists of $\min\{|N_{G_w}(x)|, k-1\}$ vertices. By Theorem~\ref{thm:fanMenger}, this implies the statement to be shown.
\end{proof}
\begin{theorem}\label{thm:3vertexRegularIs3edgeRegular}
Let $k\in\{0, 1, 2, 3\}$ and let $G$ be a $k$-connected graph in which two vertices are connected by $k+1$ edge-disjoint paths. Then there are two vertices in $G$ that are connected by $k+1$ independent paths.
\end{theorem}
\begin{proof}

We prove the claim by induction on $k$. If $k=0$ the statement of the lemma is trivially fulfilled. So let $k\in\{1,2,3\}$ and assume the statement holds true for $k-1$. Recall that Lemma~\ref{lem:parallelEdges} provides us with $k+1$ independent paths if the graph $G$ contains parallel edges. Thus we assume that $G$ is simple. Consider two vertices $v,w\in V(G)$ that are connected by $k+1$ edge-disjoint paths in $G$. 

If $v$ and $w$ are adjacent, the graph $G-vw$ is $(k-1)$-connected and contains $k$ edge-disjoint $v$-$w$ paths. By induction, in $G-vw$, there are $k$ independent $v$-$w$ paths. Together with $vw$ this yields $k+1$ independent $v$-$w$ paths in $G$.

Assume now that $v$ and $w$ are not adjacent. If $v$ and $w$ are connected by $k+1$ independent paths, there is nothing left to prove. So, by Menger's Theorem, we may assume that there is a $v$-$w$ separator containing $k$ vertices. We consider such a $v$-$w$ separator that is closest to $v$. More precisely, we choose $S$ with $|S|=k$ such that the only $v$-$S$ separator with at most $k$ vertices is $S$ itself. Furthermore, let $C$ be the component of $G-S$ that contains $v$ and define the subgraphs
\begin{equation*}
    G_v:=\big(V(C)\cup S,E(C)\cup E(C,S)\big)\quad\text{and}\quad G_w:=G-V(C),
\end{equation*}
which are shown in Figure~\ref{fig:3vertexRegularIs3edgeRegular} for the case $k=3$. Regard the $v$-$S$ subpaths of $k+1$ edge-disjoint $v$-$w$ paths. Two of these subpaths have to contain a common vertex of $S$, say $x$. Then $x$ has two neighbors in $G_v$ as well as in $G_w$. One of the neighbors in $G_v$, say $x^\prime$, is not $v$ as we obtain two parallel edges otherwise. It also holds true that $x^\prime\notin S$ as by definition $G_v$ does not contain edges joining vertices in $S$. Thus, $X:=\{x^\prime\}\cup S$ contains indeed $k+1$ elements. By the choice of $S$, the vertex set $X$ can only be separated from $v$ by at least $k+1$ vertices. So Theorem~\ref{thm:fanMenger} says that there is a $v$-$X$ fan with $k+1$ paths in $G_v$. No path of the $v$-$X$-fan contains $x^\prime x$. As a result, there exits a set of $k+1$ independent paths of which two connect $v$ and $x$ and the remaining (if any) connect to distinct vertices in $S\setminus\{x\}$. The situation is illustrated in Figure~\ref{fig:3vertexRegularIs3edgeRegular} for the case $k=3$.

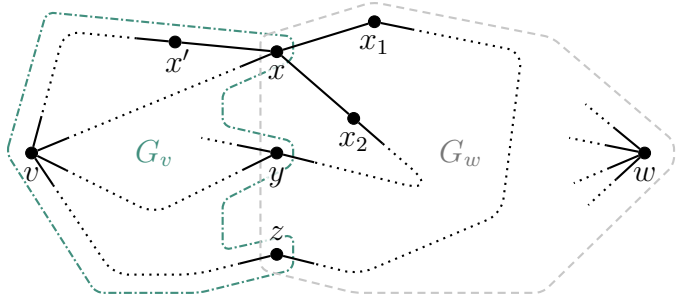
\begin{figure}
\centering
\begin{tikzpicture}[scale=1.075]
        \begin{scope}[shift={(-0.25,0)}]
    \draw[chroma4green, thick, densely dashdotted, scale=1.15, rounded corners=1mm] (-0.05,0) -- (0.375,1.51) -- (3,1.25) -- (3,1) -- (2.25,0.6675) -- (2.25,0.2970) -- (3,0.125) -- (3,-0.125) -- (2.25,-0.51) -- (2.25, -1.0425) -- (3,-0.8750) -- (3,-1.25) -- (2,-1.5) -- (0.9,-1.5) -- cycle;
    \end{scope}
    \begin{scope}[shift={(-0.65,0)}]
    \draw[chroma4gray, thick, densely dashed, scale=1.15, rounded corners=1mm] (3,1.25) -- (4.1942,1.6194) -- (6.1,1.24) -- (7.4,0.125) -- (7.4,-0.125) -- (6,-1.5) -- (4,-1.5) -- (3,-1.25) -- cycle;
    \end{scope}
    
    \draw[draw=white,fill=white, opacity=0.75] (2.88,0.9) rectangle (3.09,1.12);
    \draw[draw=white,fill=white, opacity=0.75] (2.88,-0.4) rectangle (3.08,-0.13);
    \draw[draw=white,fill=white, opacity=0.75] (2.84,-1.12) rectangle (3.09,-0.9);

    \node[draw=black, fill=black, inner sep=0.55mm, circle, label={[label distance=-0.25ex]270:$v$}] (v) at (0,0) { };
    \node[draw=black, fill=black, inner sep=0.55mm, circle, label={[label distance=-0.25ex]270:$w$}] (w) at (7.5,0) { };
    \node[draw=black, fill=black, inner sep=0.55mm, circle, label={[label distance=-0.25ex]270:$x$}] (x) at (3,1.25) { };
    \node[draw=black, fill=black, inner sep=0.55mm, circle, label={[label distance=-0.25ex]270:$y$}] (y) at (3,0) { };
    \node[draw=black, fill=black, inner sep=0.55mm, circle, label={[label distance=-0.25ex]90:$z$}] (z) at (3,-1.25) { };
    \node[draw=black, fill=black, inner sep=0.55mm, circle, label={[label distance=-0.75ex]270:$\,x^\prime$}] (x') at (1.7557,1.3696) { };
    \node[draw=black, fill=black, inner sep=0.55mm, circle, label={[label distance=-0.25ex]270:$x_1$}] (x1) at (4.1942,1.6194) { };
    \node[draw=black, fill=black, inner sep=0.55mm, circle, label={[label distance=-0.25ex]270:$x_2$}] (x2) at (3.9407,0.4269) { };
    
    \draw[black, thick] (x') to (x);
    \draw[black, thick] (x) to (x1);
    \draw[black, thick] (x) to (x2);
    
    \draw[black, thick] (v) to (0.4351,0.1813);
    \draw[black, thick] (x) to (2.5649,1.0687);
    \draw[black, thick, dotted] (0.4351,0.1813) to (2.5649,1.0687);

    \draw[black, thick] (x') to (1.2865,1.4147);
    \draw[black, thick] (v) to (0.1215,0.4555);
    \draw[black, thick, dotted, rounded corners=2mm] (1.2865,1.4147) to (0.4,1.5) to (0.1215,0.4555);
    
    \draw[black, thick] (v) to (0.2615,-0.3922);
    \draw[black, thick] (z) to (2.5427,-1.3643);
    \draw[black, thick, dotted, rounded corners=2mm] (0.2615,-0.3922) to (1,-1.5) to (2,-1.5) to (2.5427,-1.3643);
    
    \draw[black, thick] (v) to (0.4216,-0.2108);
    \draw[black, thick] (y) to (2.5784,-0.2108);
    \draw[black, thick, dotted, rounded corners=2mm] (0.4216,-0.2108) to (1.5,-0.75) to (2.5784,-0.2108);
    \draw[black, thick] (y) to (2.5380,0.0937);
    \draw[black, thick, dotted, rounded corners=2mm] (2.5380,0.0937) to (2.0760,0.1873);
    
    \draw[black, thick] (x2) to (4.2955,0.1165);
    \draw[black, thick] (y) to (3.4573,-0.1143);
    \draw[black, thick, dotted, rounded corners=5mm] (4.2955,0.1165) to (5,-0.5) to (3.4573,-0.1143);
    
    \draw[black, thick] (x1) to (4.6560,1.5249);
    \draw[black, thick] (z) to (3.4573,-1.3643);
    \draw[black, thick, dotted, rounded corners=2mm] (4.6560,1.5249) to (6,1.25) to (5.75,-0.75) to (4,-1.5) to (3.4573,-1.3643);
    
    \draw[black, thick] (w) to (7.1379,0.3018);
    \draw[black, thick, dotted, rounded corners=2mm] (7.1379,0.3018) to (6.7757,0.6036);
    \draw[black, thick] (w) to (7.0380,0.0937);
    \draw[black, thick, dotted, rounded corners=2mm] (7.0380,0.0937) to (6.5760,0.1873);
    \draw[black, thick] (w) to (7.0667,-0.1857);
    \draw[black, thick, dotted, rounded corners=2mm] (7.0667,-0.1857) to (6.6334,-0.3714);
    \draw[black, thick] (w) to (7.1551,-0.3213);
    \draw[black, thick, dotted, rounded corners=2mm] (7.1551,-0.3213) to (6.8102,-0.6426);
    
    \node[chroma4green] at (1.5,0) {$G_v$};
    \node[chroma4gray!60!black] at (5.25,0) {$G_w$};
\end{tikzpicture}
\caption{The structure of the graph in the proof of Theorem~\ref{thm:3vertexRegularIs3edgeRegular}.}
\label{fig:3vertexRegularIs3edgeRegular}
\end{figure}

Since $k\leq 3$ and $|N_{G_w}(x)|\geq 2$, it holds true that $\min\{|N_{G_w}(x)|, k-1\} = k-1$. Thus, by Lemma~\ref{lem:fanInKseparator}, in $G_w$ there exists an $x$-$(S\setminus\{x\})$ fan consisting of $k-1$ paths. Together with the aforementioned $k+1$ independent $x$-$S$ paths this yields $k+1$ independent $v$-$x$ paths in $G$.
\end{proof}

Note that the proof fails for $k>3$, as in this case, we may not assume that $\min\{|N_{G_w}(x)|, k-1\}=k-1$ and therefore the number of independent paths provided by Lemma~\ref{lem:fanInKseparator} is not enough to obtain the desired paths. This is, however, as expected, since Figure~\ref{fig:vertexRegularNotEdgeRegular} already showed us a counterexample for $k=4$.

\begin{corollary}\label{cor:3edge3vertex}
A graph is \kvertexConnected{3} if and only if it is 3-connected and \kedgeConnected{3}.
\end{corollary}
\begin{proof}
We obtain this result by combining Theorems~\ref{lem:kEdgeConnectedIsVertexConnected} and~\ref{thm:3vertexRegularIs3edgeRegular}.
\end{proof}

\section{Constructing uniformly 3-connected graphs}\label{sec:construction}

In this section, we characterize the class of \kvertexConnected{3} graphs constructively. By Corollary~\ref{cor:3edge3vertex} this is exactly the class of \kedgeConnected{3} graphs that are also $3$-connected. As we mentioned in the introduction, Kingsford and Mar\c{c}ais~\cite{kingsford2009synthesis} characterize \kedgeConnected{3} graphs. However, their construction creates separators containing exactly two vertices. Thus, their characterization may not be applied to \kvertexConnected{3} graphs directly.

We regard the following recursively defined class of simple graphs $\mathcal{G}$. We prove this class to exactly contain all \kvertexConnected{3} graphs in Theorem~\ref{thm:classG}. Recall that $\mathcal{G}$ has to contain only simple graphs by Lemma~\ref{lem:vertexRegularSimple}.
\begin{enumerate}
    \item If $G$ is $3$-connected and $3$-regular, then $G\in\mathcal{G}$
    
    \item If $G_1, G_2\in\mathcal{G}$ with vertices $v_1\in V(G_1)$ and $v_2\in V(G_2)$ with $N(v_1)=\{x_1,x_2,x_3\}$ and $N(v_2)=\{y_1,y_2,y_3\}$, then
    \begin{equation*}
        (G_1-v_1)\cup (G_2-v_2) + x_1y_1 + x_2y_2 + x_3y_3\in \mathcal{G}.
    \end{equation*}
    We denote the set of all graphs obtained from $G_1$ and $G_2$ in this way by $G_1\bridge G_2$. An example of operation~$\bridge$ is depicted in Figure~\ref{fig:bridgeOperation}.
    \item If $G\in\mathcal{G}$ with distinct vertices $v, w, x\in V(G)$, $vw\in E(G)$, $\deg(y)=3$ for all $y\in V(G)\setminus\{x\}$, and $u\notin V(G)$, then
    \begin{equation*}
        G+u-vw+uw+uv+ux\in\mathcal{G}.
    \end{equation*}
    We denote the set of all graphs obtained from $G$ in this way by $\spoke(G)$. An example of operation~$\spoke$ is depicted in Figure~\ref{fig:spokeOperation}.
\end{enumerate}

\begin{figure}
\centering
\begin{tikzpicture}[scale=1.25]
    \node[draw=gray, fill=white, inner sep=0.55mm, circle, label={[label distance=-0.75ex]315:$v_1$}] (v1) at (1,0.5) { };
    \node[draw=black, fill=black, inner sep=0.55mm, circle, label={[label distance=-0.5ex]90:$x_1$}] (x1) at (0,2) { };
    \node[draw=black, fill=black, inner sep=0.55mm, circle, label={[label distance=-0.75ex]135:$x_2$}] (x2) at (-0.25,0.75) { };
    \node[draw=black, fill=black, inner sep=0.55mm, circle, label={[label distance=-0.25ex]270:$x_3$}] (x3) at (0,0) { };
    
    \node[draw=gray, fill=white, inner sep=0.55mm, circle, label={[label distance=-0.5ex]180:$v_2$}] (v2) at (2,1.5) { };
    \node[draw=black, fill=black, inner sep=0.55mm, circle, label={[label distance=-0.5ex]90:$y_1$}] (y1) at (3,2) { };
    \node[draw=black, fill=black, inner sep=0.55mm, circle, label={[label distance=-0.75ex]45:$y_2$}] (y2) at (3.5,1.25) { };
    \node[draw=black, fill=black, inner sep=0.55mm, circle, label={[label distance=-0.25ex]270:$y_3$}] (y3) at (3,0) { };
    
    \draw[gray, thick, densely dotted] (v1) to (x1);
    \draw[gray, thick, densely dotted] (v1) to (x2);
    \draw[gray, thick, densely dotted] (v1) to (x3);
    
    \draw[gray, thick, densely dotted] (v2) to (y1);
    \draw[gray, thick, densely dotted] (v2) to (y2);
    \draw[gray, thick, densely dotted] (v2) to (y3);
    
    \draw[chroma4green, thick, dashed] (x1) to (y1);
    \draw[chroma4green, thick, dashed] (x2) to (y2);
    \draw[chroma4green, thick, dashed] (x3) to (y3);
    
    \draw[black, thick] (x1) to (x2);
    
    \draw[black, thick] (y1) to (y2);
    \draw[black, thick] (y2) to (y3);
    
    \draw[black, thick] (y1) to (3.333,2);
    \draw[black, thick, dotted] (3.333,2) to (3.667,2);
    
    \draw[black, thick] (y2) to (3.787,1.084);
    \draw[black, thick, dotted] (3.787,1.084) to (4.074,0.918);
    
    \draw[black, thick] (y3) to (3.287,0.166);
    \draw[black, thick, dotted] (3.287,0.166) to (3.574,0.333);
    
    \draw[black, thick] (x3) to (-0.333,0);
    \draw[black, thick, dotted] (-0.333,0) to (-0.667,0);
    \draw[black, thick] (x3) to (-0.287,0.166);
    \draw[black, thick, dotted] (-0.287,0.166) to (-0.574,0.333);
    
    \draw[black, thick] (x2) to (-0.583,0.75);
    \draw[black, thick, dotted] (-0.583,0.75) to (-0.916,0.75);
    
    \draw[black, thick] (x1) to (-0.287,1.834);
    \draw[black, thick, dotted] (-0.287,1.834) to (-0.574,1.667);
    
    \node at (1.5,0.75) {\textcolor{chroma4green}{$\bridge$}};
    
    \node at (-1.5,0.75) {$G_1$};
    \node at (4.5,0.75) {$G_2$};

\end{tikzpicture}
\caption{Connecting two graphs by the $\bridge$ operation.}
\label{fig:bridgeOperation}
\end{figure}
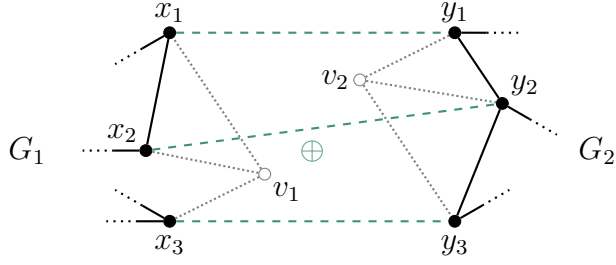

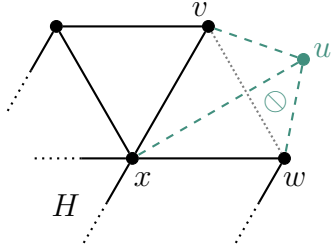
\begin{figure}
\centering
\begin{tikzpicture}[scale=0.25]
    \node[draw=black, fill=black, inner sep=0.55mm, circle] (v) at (4,7) { };
    \node[draw=black, fill=black, inner sep=0.55mm, circle] (w) at (8,0) { };
    \node[draw=black, fill=black, inner sep=0.55mm, circle] (x) at (0,0) { };
    \node[draw=black, fill=black, inner sep=0.55mm, circle] (y) at (-4,7) { };
    \node[draw=chroma4green, fill=chroma4green, inner sep=0.55mm, circle] (u) at (9,5.25) { };
    
    \draw[black, thick] (x) to (y);
    \draw[black, thick] (x) to (v);
    \draw[black, thick] (x) to (w);
    \draw[black, thick] (y) to (v);
    
    \draw[black, thick] (y) to (-5.333,4.667);
    \draw[black, thick, dotted] (-5.333,4.667) to (-6.667,2.334);
    \draw[black, thick] (x) to (-1.333,-2.333);
    \draw[black, thick, dotted] (-1.333,-2.333) to (-2.667,-4.667);
    \draw[black, thick] (x) to (-2.667,0);
    \draw[black, thick, dotted] (-2.667,0) to (-5.333,0);
    \draw[black, thick] (w) to (6.667,-2.333);
    \draw[black, thick, dotted] (6.667,-2.333) to (5.333,-4.667);
    
    \draw[gray, thick, densely dotted] (v) to (w);
    \draw[chroma4green, thick, dashed] (u) to (v);
    \draw[chroma4green, thick, dashed] (u) to (w);
    \draw[chroma4green, thick, dashed] (u) to (x);
    
    \node at (-3.5,-2.5) {$H$};
    \node at (0.55,-1.1) {$x$};
    \node at (3.55,7.9) {$v$};
    \node at (8.55,-1.1) {$w$};
    \node at (10,5.75) {\textcolor{chroma4green}{$u$}};
    %\node at (5.75,2) {\textcolor{chroma4green}{$\spoke$}};
    \node at (7.6,3) {\textcolor{chroma4green}{$\spoke$}};
\end{tikzpicture}
\caption{Enlarging a graph by the $\spoke$ operation.}
\label{fig:spokeOperation}
\end{figure}

We know that all $3$-connected, $3$-regular graphs are \kvertexConnected{3}. However, as is illustrated by Figure~\ref{fig:vertexRegularNotEdgeRegular}, not all \kvertexConnected{3} graphs are $3$-regular. Tutte~\cite{tutte1966connectivity} beautifully characterized all $3$-connected, $3$-regular graphs as the ones obtainable from $K_4$ by repeated \emph{joining} of edges. Here, joining two edges means to subdivide the edges and to join the arising vertices. In some sense the operation $\spoke$ is a variant of that operation, as we may interpret it as joining an edge to a vertex. As a first step towards proving that $\mathcal{G}$ is the class of all \kvertexConnected{3} graphs, we show that the operations $\bridge$ and $\spoke$ preserve uniform~$3$-connectivity. To this end, we call a $3$-cut in a graph \emph{degenerate} if one side of the cut consists of exactly one vertex.

\begin{lemma}
\label{lem:oplus_preserves_connectivity}
Let $G_1, G_2$ be simple graphs and $H\in G_1\bridge G_2$. Then $G_1$ and $G_2$ are \kvertexConnected{3} if and only if $H$ is \kvertexConnected{3}.
\end{lemma}
\begin{proof}
Let $v_1\in V(G_1)$ and $v_2\in V(G_1)$ with $N(v_1)=\{x_1,x_2,x_3\}$ and $N(v_2)=\{y_1,y_2,y_3\}$ such that $H= (G_1-v_1)\cup (G_2-v_2) + x_1y_1 + x_2y_2 + x_3y_3$. 

First assume that $G_1$ and $G_2$ are \kvertexConnected{3}. Suppose for the sake of contradiction that $H$ contains a $2$-separator $S$ that separates the vertices $v,w\in V(H)$. If $v,w\in V(G_l)$ for some $l\in\{1,2\}$, say $v, w\in V(G_1)$, then in $G_1-S$, there exists a path $P$ connecting $v$ and $w$ as $G_1$ is $3$-connected. If $P$ does not contain the vertex $v_1$, it also exists in $H-S$ yielding a contradiction to $S$ separating $v$ and $w$ in $H$. If $P$ contains $v_1$, then $S\subseteq V(G_1)$ and we may replace the subpath $x_{i}v_1x_{j}$ in $P$ by a path $x_{i}y_{i}P^\prime y_{j}x_{j}$, where $P^\prime$ is a $y_{i}$-$y_{j}$ path in $G_2-v_2$. The resulting path also exists in $H-S$ again yielding a contradiction. Thus, we may assume that $v\in V(G_1)$ and $w\in V(G_2)$. In this case either $v$ and $v_1$ are separated by $S$ in $G_1$ or $w$ and $v_2$ are separated by $S$ in $G_2$. In both cases we get a contradiction to $G_1$ or $G_2$ being $3$-connected.
To see that $H$ is also \kvertexConnected{3}, suppose for the sake of contradiction that there exist four independent paths between two vertices $v,w\in V(H)$. By the definition of the operation $\bridge$, we may assume that $v,w\in V(G_1)$ and at most one of the four independent paths, say $P$, may use edges from $E(G_2)$. We may replace the subpath of $P$ that uses edges not contained in $E(G_1)$ by a path $x_{i}v_1x_{j}$ yielding a path in $G_1$. As this path is still independent to the other three paths, we found four independent $v$-$w$ paths in $G_1$, which is a contradiction.

Let now $H$ be \kvertexConnected{3}. We begin by proving that $G_1$ and $G_2$ are $3$-connected. As the cases are symmetric it suffices to show that $G_1$ is $3$-connected. To this end let $v, w\in V(G_1)$ be chosen arbitrarily. If $v_1\in\{v,w\}$, say $v_1=v$, we may use three independent $w$-$y_1$ paths in $H$ to define three independent $w$-$v$ paths in $G_1$. So let $v_1\notin\{v,w\}$ and consider three independent $v$-$w$ paths $P_1$, $P_2$, $P_3$ in $H$. Only one of the paths may use edges from $\{x_1y_1, x_2y_2, x_3y_3\}$. Thus, similarly to above we can define three independent $v$-$w$ paths in $G_1$ replacing the subpath not contained in $G_1$ by an adequately chosen path $x_{i}v_1x_{j}$. Thus, between any two vertices in $G_1$ there exist three independent paths and $G_1$ is $3$-connected by Menger's Theorem.
Now suppose there exist $v, w\in V(G_1)$ that are connected by four independent paths. As $v_1$ is of degree $3$, we have $v_1\notin\{v,w\}$ and only one of the four paths may touch $v_1$. Thus, again we may define four independent paths in $H$ yielding a contradiction.
\end{proof}
\begin{lemma}\label{lem:spoke_preserves_connectivity}
Let $G$ be a simple graph and $H\in\spoke(G)$.
\begin{enumerate}
    \item If $G$ is \kvertexConnected{3}, then $H$ is \kvertexConnected{3}.\label{claim:2}
    \item If every $3$-cut in $H$ is degenerate and  $H$ is \kvertexConnected{3}, then $G$ is \kvertexConnected{3}.\label{claim:3}
\end{enumerate}
\end{lemma}
\begin{proof}
Let $v, w\in V(G)$, $N(v)=\{w, x, y\}$, $N(w)=\{v,x,z\}$, such that $H=G+u-vw+uw+uv+ux$ with $u\notin V(G)$.

We begin by proving Claim~\ref{claim:2}. So let $G$ be \kvertexConnected{3}. Note that $H$ contains exactly one vertex of degree larger than $3$. As a consequence, the maximum number of independent paths between any pair of vertices in $H$ is at most three. Thus, it suffices to prove that $H$ is $3$-connected. Let $a,b\in V(G)$ be chosen arbitrarily. We show that there exist three independent $a$-$b$ paths in $H$. Since $G$ is $3$-connected, there exist three independent $a$-$b$ paths in $G$. If the edge $vw$ does not occur in one of these paths, there is nothing to show, as $vw$ is the only element in $G$ that does not exist in $H$. If $vw$ occurs in one of the paths, we may safely replace the edge $vw$ by $vuw$, as the vertex $u$ and the edges $vu$ and $uw$ do not exist in $G$. Thus, there is no $2$-separator in $H$ that separates two vertices in $V(G)$. Now suppose there exists a $2$-separator $S$ separating $u$ and some other vertex $a\in V(H)$. As $u$ has three neighbors, in $G-S$ it is contained in a component with more than one vertex. This implies that $S$ also separates two vertices $x, y\in V(G)$ in $H$, which is not possible by previous arguments. Thus, $H$ is $3$-connected and as a consequence also \kvertexConnected{3}.

We now turn to Claim~\ref{claim:3}. So assume that $H$ is \kvertexConnected{3} and only contains degenerate $3$-cuts. Again we only need to show that $G$ is $3$-connected as at most one vertex has a degree larger than $3$ and therefore between any pair of vertices there are not more than three independent paths. We begin by showing that $x$ may not be separated from any other vertex in $G$ by removing two vertices. Suppose in $G-w_1-w_2$ there is no $x$-$a$ path for some $a\in V(H)$ and $w_1,w_2\in V(G)\setminus\{x, a\}$. Denote by $G_x$ the component of $x$ and $G_a$ the component of $a$ in $G-w_1-w_2$. See Figure~\ref{fig:spoke_preserves_proof} for an example of this separation. For $i\in\{1,2\}$ the vertex $w_i$ is of degree $3$ and we therefore have
\begin{equation*}
    \min\big\{E\big(\{w_i\},V(G_x)\big), E\big(\{w_i\}, V(G_a)\big)\big\}=1.
\end{equation*}
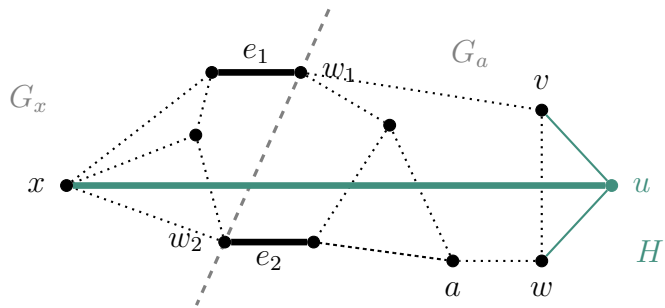
\begin{figure}
    \centering
    \begin{tikzpicture}
        \node[vertex] (x) {};
        \labelLeft{x}{$x$}
        \node[vertex, right=7cm of x, chroma4green] (u) {};
        \labelRight{u}{\textcolor{chroma4green}{$u$}}
        \node[vertex, left=.75cm of u, yshift=1cm] (v) {};
        \labelAbove{v}{$v$}
        \node[vertex, left=.75cm of u, yshift=-1cm] (w) {};
        \labelBelow{w}{$w$}
        
        \node[vertex, left=3cm of v, yshift=.5cm] (w1) {};
        \labelRight{w1}{$w_1$}
        \node[vertex, left=4cm of w, yshift=.25cm] (w2) {};
        \labelLeft{w2}{$w_2$}
        \node[vertex, left=1cm of w] (y) {};
        \labelBelow{y}{$a$}
        \node[vertex, left=1cm of w1] (a1) {};
        \node[vertex, right=1cm of w2] (a2) {};
        \node[above=.75cm of x, xshift=-.5cm, gray] {$G_x$};
        \node[right=1.75cm of w1, yshift=.25cm, gray] {$G_a$};
        \node[below=.5cm of u, xshift=.5cm, chroma4green] {$H$};
        
        \node[vertex, right=1cm of w1, yshift=-.7cm] (1) {};
        \node[vertex, above=.5cm of x, xshift=1.7cm] (2) {};
        
        \draw[edge, dotted] (x) -- (a1) -- (2) -- (w2) -- (x);
        \draw[edge, dotted] (x) -- (2);
        \draw[edge, dotted] (w1) -- (v) -- (w) -- (y) (w1) -- (1) -- (a2);
        \draw[edge, dotted] (y) -- (1);
        \draw[edge, dotted] (a2) -- (y);
        \draw[edge, dotted] (y) -- (a2);
        
        \draw[edge, line width=2.5pt] (a1) -- (w1) node[midway, yshift=.25cm] {$e_1$};
        \draw[edge, line width=2.5pt] (a2) -- (w2) node[midway, yshift=-.25cm] {$e_2$};
        \begin{pgfonlayer}{background}
        \draw[dashed, very thick, gray, shorten >=-1.0cm,shorten <=-1.0cm] (w1) -- (w2);
        \end{pgfonlayer}
        \draw[edge, chroma4green] (v) -- (u) -- (w);
        \draw[edge, line width=2.5pt, chroma4green] (x) -- (u);

    \end{tikzpicture}
    \caption{Supposed $2$-separation of $x$ and $y$ in the proof of Lemma~\ref{lem:spoke_preserves_connectivity}. The thick edges indicate a $3$-cut. Edges are dotted to indicate that the depicted picture is not the only possible situation in the proof.}
    \label{fig:spoke_preserves_proof}
\end{figure}%
Denote by $e_i$ the single edge connecting $w_i$ to the respective component $G_x$ or $G_a$. Then $\{e_1, e_2\}$ is a $2$-cut in $G$. If $vw\notin\{e_1, e_2\}$ the set $E^\prime=\{e_1, e_2, xu\}$ separates $x$ and $a$ in $H$, which is illustrated in Figure~\ref{fig:spoke_preserves_proof}. As $H$ only contains degenerate $3$-cuts, all of the edges in $E^\prime$ are incident to the same degree $3$ vertex. As $x$ is of degree larger than $3$ in $H$ and $e_1$ and $e_2$ are not incident to $u$, this yields a contradiction. If $vw\in\{e_1,e_2\}$, then $E^\prime=\{e_1, e_2, xu\}\setminus\{vw\}\cup\{vu\}$ separates $x$ and $a$ in $H$. Again, one of the edges in $E^\prime$ is not incident to $u$ and therefore $E^\prime$ is not degenerate, which is a contradiction. Thus, we have shown that $x$ cannot be separated from any other vertex in $G$ by a $2$-separator. However, note that $x$ may also not be contained in a $2$-separator in $G$, as this separator would also be a separator in $H$. This yields a contradiction and we conclude that $G$ is $3$-connected, which completes the proof.
\end{proof}

Note that in Claim~\ref{claim:3} of Lemma~\ref{lem:spoke_preserves_connectivity}, the condition that $H$  only contains degenerate $3$-cuts is in fact necessary: Any graph in $W_4\bridge W_5$ serves as a counterexample if the condition is dropped, where $W_n$ denotes a wheel graph on $n$ vertices.

\begin{theorem}\label{thm:classG}
The class $\mathcal{G}$ contains exactly all \kvertexConnected{3} graphs.
\end{theorem}
\begin{proof}
First, we observe that any simple $3$-regular $3$-connected graph is \kvertexConnected{3}. Furthermore, by Lemma~\ref{lem:oplus_preserves_connectivity} and~\ref{lem:spoke_preserves_connectivity}, both operations $\bridge$ and $\spoke$ preserve uniform $3$-connectivity. Therefore any graph in $\mathcal{G}$ is \kvertexConnected{3}.

We now turn to the other direction and prove that any \kvertexConnected{3} graph $G$ is contained in $\mathcal{G}$. We prove the claim by induction on $|V(G)|$. If $|V(G)|=4$, we have $G=K_4$ and the claim holds true. So assume $|V(G)|\geq 5$. We distinguish two cases.

In the first case, assume that $G$ contains a non-degenerate $3$-cut $E^\prime=\{e_1, e_2, e_3\}$. We claim that in this case no vertex in $G$ is incident to more than one edge in $E^\prime$. If a vertex is incident to all edges in $E^\prime$ but not of degree $3$, then removing this vertex from the graph separates the graph, which is not possible as $G$ is $3$-connected. Now suppose a vertex $v$ is incident to exactly two edges of $E^\prime$, say $e_1$ and $e_2$. Let $x$ be the vertex incident to $e_3$ that is not contained in the same component as $v$ in $G-E^\prime$. As $x$ is not incident to at least one of the edges $e_1$ and $e_2$ and $v$ is not incident to $e_3$, the graph $G-x-v$ is not connected. This contradicts the fact that $G$ is $3$-connected. Thus, we may assume that the endvertices of the edges in $E^\prime$ are distinct. Denote by $X$ and $Y$ the two components of $G-E^\prime$. Furthermore, denote for $i\in\{1,2,3\}$ by $x_i$ the distinct endvertices of $e_i$ in $X$ and by $y_i$ the distinct endvertices of $e_i$ in $Y$. Then for vertices $v_1,v_2\notin V(G)$ it is $G\in G_1\bridge G_2$, where $G_1=X+v_1+x_1v_1+x_2v_1+x_3v_1$ and $G_2=Y+v_2+y_1v_2+y_2v_2+y_3v_2$. Lemma~\ref{lem:oplus_preserves_connectivity} then implies that $G_1$ and $G_2$ are \kvertexConnected{3}. Both graphs $G_1$ and $G_2$ contain fewer vertices than $G$. It follows by induction that $G_1, G_2\in\mathcal{G}$ and therefore $G\in\mathcal{G}$.

In the second case, assume that every $3$-cut in $G$ is degenerate. If $G$ does not contain a vertex of degree at least $4$, it is $3$-regular and therefore contained in $\mathcal{G}$. Next, we show that at most one vertex in $G$ is of degree larger than $3$. Suppose there are two vertices of degree larger than $3$, say $v$ and $w$. As $G$ is \kvertexConnected{3}, by Corollary~\ref{cor:3edge3vertex}, there exists a $3$-cut separating $v$ and $w$. As $v$ and $w$ are of degree $4$, all edges of this cut must be incident to some degree $3$ vertex $x\notin\{v,w\}$. Then $x$ separates $v$ and $w$ in $G$ which is not possible. Thus we denote by $x$ the unique vertex in $G$ that has degree at least $4$. Let $u$ be a neighbor of $x$ and let $N(u)=\{x, v, w\}$. Suppose $vw\in E(G)$ which implies that $uvwu$ is a triangle. Then $E^\prime\coloneqq E(\{u,v,w\}, V(G)\setminus\{u,v,w\})$ contains exactly three edges and separates $\{u,v,w\}$ from $x$. As $ux\in E^\prime$ and $G$ only contains degenerate~$3$-cuts, this is a contradiction. Thus, we have $vw\notin E(G)$ which implies that $H\coloneqq G-u+vw$ is simple. We obtain that $G \in \spoke(H)$ and by Lemma~\ref{lem:spoke_preserves_connectivity} that $H$ is \kvertexConnected{3}. As $H$ contains fewer vertices than $G$, by induction, $H$ is contained in $\mathcal{G}$ and, thereby, so is $G$.
\end{proof}

We have seen that the operations $\bridge$ and $\spoke$ give us the means to construct all \kvertexConnected{3} graphs. In the next section, we see an example for how this construction may be used in order to derive further structural properties of \kvertexConnected{3} graphs.

\section{Vertices of minimum degree}\label{sec:degrees}
In this section, we provide a tight lower bound on the number of vertices of minimum degree in \kvertexConnected{3} graphs. We denote this parameter for a graph~$G$ by
\begin{equation*}
    \nu(G)\vcentcolon=\big|\big\{v\in V(G):\deg(v)=\min_{\mathclap{v\in V(G)}}\deg(v)\big\}\big|.
\end{equation*}

\begin{theorem}\label{thm:boundNu3}
Let~$G$ be a \kvertexConnected{3} graph on $n$ vertices. Then
\begin{equation*}
    \nu(G) \ge \frac{2n+2}{3}.
\end{equation*}
\end{theorem}
\begin{proof} First note that for $3$-regular graphs the statement of the theorem is clearly fulfilled. The only \kvertexConnected{3} graph with $n\le 4$ is the complete graph for which our assertion is correct as it is $3$-regular.

In the following, we argue by induction on $n$. Let $G$ be a \kvertexConnected{3} graph that is not 3-regular. By Theorem~\ref{thm:classG}, $G$ can be decomposed in accordance with operation $\bridge$ or operation $\spokeStandAlone$ into graphs from the class~$\mathcal{G}$, which are again \kvertexConnected{3}.

Let us assume first that $G\in\spoke(H)$ for a graph $H$ from the class $\mathcal{G}$. So we have
\begin{equation*}
    G=H+u-vw+uw+uv+ux
\end{equation*}
for vertices $u\notin V(H)$ and $v,w,x\in V(H)$ with $vw\in E(H)$ and $x$ being the only vertex in $H$ whose degree might be larger than $3$. This construction leads to $\deg(u)=3$ and preserves the degrees of the vertices $v$ and $w$, as it is also illustrated by Figure~\ref{fig:spokeOperation}. So $x$ is the only vertex in $G$ with $\deg(x)>3$. As a consequence, we get in this case
\begin{equation*}
    \nu(G) = n-1 \ge \frac{2n+2}{3}.
\end{equation*}
Let us now assume that $G\in G_1 \bridge G_2$ for two graphs $G_1$ and $G_2$ from the class $\mathcal{G}$. So we have
\begin{equation*}
    G=(G_1-v_1)\cup(G_2-v_2) + x_1y_1 + x_2y_2 + x_3y_3
\end{equation*}
for $v_1\in V(G_1)$ and $v_2\in V(G_2)$ with $N(v_1)=\{x_1,x_2,x_3\}$ and $N(v_2)=\{y_1,y_2,y_3\}$. By this construction, we have~$\deg(v_1) = \deg(v_2) = 3$. Furthermore, operation~$\bridge$ preserves the degrees of all vertices from $G_1$ and $G_2$ that are present in $G$, which is illustrated by Figure~\ref{fig:bridgeOperation}. So only the absence of the vertices~$v_1$ and $v_2$ themselves affects the parameter~$\nu(G)$. Consequently, we get $\nu(G) = \nu(G_1)+\nu(G_2)-2$. Moreover, denoting the number of vertices of $G_1$ and $G_2$ by $n_1$ and $n_2$, respectively, we have $n=n_1+n_2-2$. Hence, the graphs $G_1$ and $G_2$ have fewer vertices than~$G$ and we conclude by induction that
\begin{equation*}
    \nu(G) \ge \frac{2n_1+2}{3} + \frac{2n_2+2}{3} - 2 = \frac{2(n_1 + n_2 - 2) +2}{3} = \frac{2n+2}{3},
\end{equation*}
which finishes the proof.\end{proof}
In fact, this theorem provides a tight bound for general $n$. To examine this, we first observe that $W_{n+1}\in\spoke(W_n)$ for all $n\geq 4$, where $W_n$ denotes the wheel graph on $n$ vertices. Any graph in $\bigoplus_{i=1}^k W_5$ attains the given bound for all $n=5+3k$ with $k\in\mathbb{N}$, any graph in $(\bigoplus_{i=1}^{k-1} W_5) \bridge W_6$ attains the bound for all $n=6+3k$ with $k\in\mathbb{N}$, and any graph in $(\bigoplus_{i=1}^{k-1} W_5) \bridge W_7$ attains our bound for all $n=7+3k$ with $k\in\mathbb{N}$. Here, for graphs $G$, $G_1$,\ldots, $G_k$ it is
\begin{align*}
\{G_1,\ldots, G_k\}\bridge G\coloneqq (G_1\bridge G)\cup (G_2\bridge G)\cup\ldots\cup (G_k\bridge G)
\end{align*}

The analogous problem for \edgeConnected graphs is solved by Kingsford and Mar{\c{c}}ais~\cite{kingsford2009vertices}. They show that $\nu(G)\ge 2$ for a \kedgeConnected{k} graph $G$ for general $k\in\mathbb{N}$. To see that this bound is best possible for an arbitrary number of vertices, consider a graph whose underlying simple graph is a path and each edge has exactly $k-1$ parallels.

\section{Conclusions and related problems}

We studied \kvertexConnected{k} graphs and their relation towards \kedgeConnected{k} graphs. In particular, we thoroughly regarded these classes for $k=3$. We provided a constructive characterization for all \kvertexConnected{3} graphs. This construction is a promising tool for proving further properties of \kvertexConnected{3} graphs. We demonstrated this by utilizing the characterization to derive a tight lower bound on the number of vertices of minimum degree. However, our construction ideas cannot be generalized to $k\geq 4$ directly. So finding a construction for general $k$ remains an interesting problem.

We also showed that any simple, $3$-connected graph in which two vertices are connected by four edge-disjoint paths also contains two vertices that are connected by four independent paths. We gave an example that the same does not hold true for $4$-connected graphs and five edge-disjoint paths. So, more generally, one may ask for sufficient conditions under which edge-disjoint paths guarantee independent paths.

\section*{Acknowledgments}
Our research was partially funded by the Deutsche Forschungsgemeinschaft (DFG, German Research Foundation) -- Project-ID 416228727 -- SFB 1410.

\bibliographystyle{plain}
\bibliography{bibliography}

\end{document}